\documentclass[12pt]{amsart}

\usepackage{amsmath, amscd, amssymb}

\newcommand{\bC}{{\mathbb C}}

\newcommand{\bZ}{{\mathbb Z}}

\newcommand{\cC}{{\mathcal C}}

\newcommand{\Mbar}{\overline{\mathcal M}}

\newcommand{\pd}{\partial}
\newcommand{\half}{\frac{1}{2}}

\newcommand{\cor}[1]{\langle {#1} \rangle}
\DeclareMathOperator{\res}{Res}
\DeclareMathOperator{\Aut}{Aut}

\newtheorem{theorem}{Theorem}[section]
\newtheorem{Theorem}{Theorem}
\newtheorem{proposition}[theorem]{Proposition}

\theoremstyle{remark}
\newtheorem{remark}{Remark}[section]

\theoremstyle{definition}

\newcommand{\bea}{\begin{eqnarray}}
\newcommand{\eea}{\end{eqnarray}}
\newcommand{\ben}{\begin{eqnarray*}}
\newcommand{\een}{\end{eqnarray*}}
\newcommand{\be}{\begin{equation}}
\newcommand{\ee}{\end{equation}}

\begin{document}

\title
{Local Mirror Symmetry for One-Legged Topological Vertex}

\author{Jian Zhou}
\address{Department of Mathematical Sciences\\Tsinghua University\\Beijing, 100084, China}
\email{jzhou@math.tsinghua.edu.cn}

\begin{abstract}
We prove the Bouchard-Mari\~no Conjecture for the framed one-legged topological vertex by deriving
the Eynard-Orantin type recursion relations from the cut-and-join equation
satisfied by the relevant triple Hodge integrals.
This establishes a version of local mirror symmetry for the local $\bC^3$ geometry with
one $D$-brane.
\end{abstract}

\maketitle

\section{Introduction}

Local  mirror symmetry relates the A-theory on an open toric Calabi-Yau threefold
with the B-theory on its mirror manifold.
Mathematically in the local A-theory one studies the local Gromov-Witten invariants.
In a series of work culminating in \cite{AKMV},
a formalism called the topological vertex based on duality with Chern-Simons theory has been developed
in the physics literature.
A mathematical theory of the toplogical vertex \cite{LLLZ} has been developed,
based on a series of earlier papers \cite{Zho1, LLZ, LLZ2}.
The B-theory in genus $0$ was originally studied by the theory of variation of Hodge structures
and Frobenius manifolds.
In higher genera,
they have been studied from various points of view,
including holomorphic anomaly equation and $tt^*$-geometry \cite{BCOV}.
Recently a new formalism  for the local B-theory on the mirror of toric Calabi-Yau threefolds has been proposed
in \cite{Mar, BKMP},
inspired by the recursion procedure of \cite{EO} discovered first in the context of matrix models.

The new formalism of the B-theory makes it possible to verify local mirror symmetry in arbitrary genera,
and this has been done in many cases in \cite{Mar} and \cite{BKMP}.
The simplest case is the one-legged framed topological vertex,
Bouchard and Mari\~no \cite{Bou-Mar} made a conjecture about it
based on the proposed new formalism of the B-theory in \cite{Mar, BKMP}.
In this paper we will present a proof of this conjecture.

Mathematically the framed topological vertex involves some special Hodge integrals associated with partitions of
positive integers.
The one-legged case appear naturally in the formal localization calculations of some open Gromov-Witten invariants
by Katz and Liu \cite{Kat-Liu}.
A closed formula for them was conjectured by Mari\~no and Vafa \cite{Mar-Vaf},
and proved mathematically  by Liu-Liu-Zhou \cite{LLZ} and Okounkov-Pandharipande \cite{Oko-Pan}.
We will refer to such Hodge integrals as triple Hodge integrals of KLMV type.

It  is well-known that a certain limit of the Mari\~no-Vafa formula is the ELSV formula
that relates linear Hodge integrals to Hurwitz numbers.
Based on this, Bouchard and Mari\~no \cite{Bou-Mar} made a corresponding conjecture for Hurwitz numbers
and linear Hodge integrals.
This conjecture has been proved by Borot-Eynard-Safnuk-Mulase \cite{BEMS}
and Eynard-Safnuk-Mulase \cite{EMS} by two different methods.
In this paper we will follow \cite{EMS} to prove the Bouchard-Mari\~no Conjecture for
triple Hodge integrals of KLMV type.
The following is our main result:

\begin{Theorem}
For triple Hodge integrals of KLMV type,
the Bouchard-Mari\~no recursion can be derived from the cut-and-join equation.
\end{Theorem}

The starting point of our proof,
as in \cite{EMS},
is the cut-and-join equation as suggested in \cite{Bou-Mar}.
This equation was originally studied for Hurwitz numbers from a combinatorial point of view \cite{Gou-Jac-Vai},
and later by a symplectic geometric point of view \cite{Li-Zha-Zhe}.
Inspired by the ELSV formula,
it was proposed by the author in \cite{Zho1} (first draft written in November 2002)
that cut-and-join equation may be used as a tool to study Hodge integrals,
in particular,
a proof of the Mari\~no-Vafa formula by establishing the cut-and-join equation
geometrically via localization on relative moduli spaces was proposed.
In collaborations with Kefeng Liu and Melissa Liu \cite{LLZ} this proposal was carried out.
Furthermore,
it was generalized in \cite{LLZ2} to obtain the two-partition Hodge integral formula conjectured
in \cite{Zho2},
and in \cite{LLLZ} to obtain the three-partition Hodge integral formula in
the mathematical theory of the topological vertex developed jointly with Jun Li.
Our proof is a slight simplification of the method in \cite{EMS} because we do not use the
Laplace transform.
After the completion of this work,
there appears a paper \cite{Che2} which gets the same result by the method of \cite{EMS}.
We believe the cut-and-join equation may play an important role in establishing
the local mirror symmetry by relating the mathematical computations in the local A-theory
with the new formalism of the local B-theory.
We will present the corresponding results for the framed topological vertex
in the two-legged and three-legged cases in a forthcoming work \cite{Zho3}.

\section{Bouchard-Mari\~no Conjecture for Triple Hodge Integrals of KLMV Type}

In this section we state and reformulate Bouchard-Mari\~no Conjecture for triple Hodge integrals of
KLMV type.

\subsection{Triple Hodge integrals of KLMV type}

For a partition $\mu = (\mu_1, \dots, \mu_{l(\mu)})$ of $d>0$,
consider triple Hodge integrals of the form:
\begin{eqnarray*}
W_{g, \mu}(a) & = & \frac{(-1)^{g+l(\mu)}}{|\Aut(\mu)|}
 (a(a+1))^{l(\mu)-1}
\prod_{i=1}^{l(\mu)}\frac{ \prod_{j=1}^{\mu_i-1} (\mu_i a+j)}{(\mu_i-1)!} \\
&& \cdot \int_{\Mbar_{g, l(\mu)}}
\frac{\Lambda^{\vee}_g(1)\Lambda^{\vee}_g(-a-1)\Lambda_g^{\vee}(a)}
{\prod_{i=1}^{l(\mu)}(1- \mu_i \psi_i)},
\een
where $\Lambda_g^{\vee}(a) = \sum_{i=0}^g (-1)^i a^{g-i} \lambda_i$.

Note these are defined for $2g-2+l(\mu) > 0$,
so that $\Mbar_{g,n}$ makes sense.
It is useful to extend the definition to the exceptional cases  $(g, l(\mu)) = (0, 1)$ and $(g, l(\mu)) = (0, 2)$
by the following conventions:
\bea
&& \int_{\Mbar_{0,1}} \frac{\Lambda_0^{\vee}(1)\Lambda_0^\vee(a) \Lambda_0^\vee(-1-a)}{1 - m \psi_1}
= m^{-2}, \label{eqn:Exception1} \\
&& \int_{\Mbar_{0,2}} \frac{\Lambda_0^{\vee}(1) \Lambda_0^\vee(a) \Lambda_0^\vee(-1-a)}{(1- m_1\psi_1)(1- m_2\psi_2)}
= \frac{1}{m_1+m_2}. \label{eqn:Exception2}
\eea

Write
\ben
&& \cor{\tau_{b_1} \cdots \tau_{b_n} T_g(a)}_g
= \int_{\Mbar_{g,n}} \prod_{i=1}^n \psi_i^{b_i} \cdot
\Lambda_g^{\vee}(1)\Lambda_g^{\vee}(a)\Lambda_g^{\vee}(-1-a).
\een
For $n\geq 1$,
following Bouchard and Mari\~no \cite{Bou-Mar}, define
\be
W_g(x_1, \dots, x_n;a) = \sum_{l(\mu) =n} z_{\mu} W_{g,\mu}(a) m_{\mu}(x_1, \dots, x_n)
\prod_{i=1}^n \frac{dx_i}{x_i},
\ee
where
\bea
&& z_\mu = |\Aut(\mu)| \cdot \prod_{i=1}^n \mu_i, \\
&& m_\mu(x_1, \dots, x_n) = \frac{1}{|\Aut(\mu)|} \sum_{\sigma \in S_n} x_{\sigma(i)}^{\mu_i}.
\eea
Then we have
\begin{multline}
W_g(x_1, \dots, x_n;a)  \\
 = (-1)^{g+n} (a(a+1))^{n-1}
\cdot \sum_{b_i=0}^{3g-3+n}   \langle \tau_{b_1} \cdots \tau_{b_n} T_g(a)  \rangle_g
\prod_{i=1}^n d \phi_{b_i}(x_i, a),
\end{multline}
for $2g-2+n > 0$,
where
\be
\phi_b(x; a)=\sum_{m=1}^{\infty}{\prod_{j=1}^{m-1} (\mu a+ j) \over (m-1)!} m^{b} x^{m}.
\ee

\subsection{Ramification point on the framed mirror curve}

Eynard-Orantin formalism \cite{EO} recursively defines a sequence of differentials on
a plane algebraic curve.
The relevant curve in this case is the framed mirror curve as suggested in \cite{BKMP}:
\be \label{eqn:Mirror}
x = y^a - y^{a+1}.
\ee
Near $(x,y) = (0,1)$,
one can invert the above equation to get:
\be
y(x)=1 - \sum_{n=1}^{\infty} \frac{\prod_{j=0}^{n-2} (na+j)}{n!} x^n.
\ee
Indeed,
set
\be \label{eqn:YinU}
y=1-u,
\ee
then by applying Lagrange inversion to the equation:
\be \label{eqn:XinU}
x = u(1-u)^a
\ee
one can get \cite{Che}:
\be
u(x)= \sum_{n=1}^{\infty} \frac{\prod_{j=0}^{n-2} (na+j)}{n!} x^n.
\ee

\begin{remark}
After a ``mirror transformation'' of the form
\begin{align}
x & \mapsto - (-1)^a x, & y & \mapsto - y,
\end{align}
equation (\ref{eqn:Mirror}) transforms into the equation:
\be
x + y^a + y^{a+1} = 0.
\ee
This is equation (4.3) for the framed mirror curve in \cite{BKMP}.
\end{remark}

The $x$-projection from this curve to $\bC$ is ramified.
We have
\begin{align}
\frac{dx}{dy} & = a y^{a-1} - (a+1)y^a, &
\frac{d^2x}{dy^2} & = a(a-1) y^{a-2} - a (a+1) y^{a-1},
\end{align}
therefore $\frac{dx}{dy} = 0$ if and only if $y = \frac{a}{a+1}$;
furthermore,
$ \frac{d^2x}{dy^2} \biggl|_{y = \frac{a}{a+1}} \neq 0$.
Hence the $x$-projection from this curve to $\bC$ is two-to-one near
\be
(x, y) = (\frac{a^a}{(a+1)^{a+1}},
\frac{a}{a+1}).
\ee
There are two points $q$ and $\tilde{q}$ on the curve near the ramification point such that
$x(q) = x(\tilde{q})$.
Write
\begin{align} \label{eqn:YinZ}
y(q) & = \frac{a}{a+1} + z, & y(\tilde{q}) & = \frac{a}{a+1} + P(z),
\end{align}
where $P(z) =  - z + o(z)$.
Then from
$$ y(q)^a - y(q)^{a+1} = y(\tilde{q})^a - y(\tilde{q})^{a+1},$$
one can find \cite{Bou-Mar}:
\be
\begin{split}
P(z) = & - z - \frac{2(a^2-1)}{3a} z^2 - \frac{4(a^2-1)^2}{9a^2} z^3 \\
& - \frac{2(a+1)^3(22a^3-57a^2+57a-22)}{135a^3}z^4 + \cdots.
\end{split}
\ee

\subsection{Bouchard-Mari\~no Conjecture}

By abuse of notations,
we will write $W(x_1, \dots, x_n;a)dx_1 \cdots dx_n$ as $W(y_1, \dots, y_n;a)$.
It was conjectured in \cite{Bou-Mar},
by making the proposal in \cite{Mar, BKMP} explicit in this situation,
that the differentials $W(y_1, \dots, y_n;a)$ can be computed recursively
by the Eynard-Orantin formalism as follows.
The initial values are
\bea
&& W_0(y;a) = \ln y(x;a) {d x \over x}, \label{eqn:W0y} \\
&& W_0 (y_1, y_2;a) = \frac{dy_1dy_2}{(y_1-y_2)^2} - {d x_1 d x_2 \over (x_1 - x_2)^2}, \label{eqn:W0y1y2}
\eea
and the recursion is given by:
\begin{multline} \label{eqn:Bou-Mar}
W_g(y_{[n]};a)
= \res_{z=0}{d E_{z}(y_1) \over \omega(z)} \\
\quad  \Big( W_{g-1} ({a \over a+1} + z,
{a \over a+1} + P(z), y_{[n]_1};a ) \\
 \quad +\sum_{\substack{g_1+g_2=g\\ A\coprod B = [n]_1}} W_{g_1}({a \over a+1} + z, y_A;a)
W_{g_2}({a \over a+1} + P(z), y_B;a)\Big),
\end{multline}
where
\bea
&& \omega(z) =  ( \ln y(z) - \ln y(P(z)) )  \cdot {d x(z) \over x(z) }, \\
&& d E_z (y_1) = {1 \over 2}\left( {1 \over y(z_1) - y(z)} - {1 \over y(z_1) - y(P(z))} \right) dy_1.
\eea
Note the differentials $\omega_z$ and $d E_z(y_1)$ are closely related to $W_0(y;a)$ and $W_0(y_1,y_2;a)$
respectively.
On the right-hand side of (\ref{eqn:Bou-Mar}),
only exceptional cases $W_0(\frac{a}{a+1}+z, y_i;a)$ and $W_0(\frac{a}{a+1}+P(z), y_i;a)$ appear.
I.e.,
$W_0(\frac{a}{a+1}+z;a)$ and $W_0(\frac{a}{a+1}+P(z);a)$ are understood as $0$ in (\ref{eqn:Bou-Mar}).
Here and below  we use the following notations:
\bea
&& [n]=\{1, \dots, n\},\\
&& [n]_i = \{1, \dots n\} - \{i\}, \\
&& [n]_{ij} = \{1, \dots, n\} - \{i, j\}.
\eea
We let $x_{[n]}$ stand for $x_1, \dots ,x_n$.

We now make the recursion relations (\ref{eqn:Bou-Mar}) more explicit.
Write $\psi_b(y;a) = \phi_b(x;a)$.
From the definitions we have
\ben
W_g(y_{[n]};a)
=(-1)^{g+n}(a(a+1))^{n-1} \sum \langle \prod_{i=1}^n \tau_{b_i} \cdot T_g(a)  \rangle_g
\prod_{i=1}^n \psi_{b_i+1}(y_i; a) \frac{dx_i}{x_i}.
\een
It follows that
\ben
&&  W_{g-1} ({a \over a+1} + z, {a \over a+1} + P(z), y_{[n]_1} ) \\
& = & (-1)^{g+n} \sum \langle \tau_b \tau_c \prod_{i=2}^n \tau_{b_i} \cdot T_{g-1}(a)  \rangle_{g-1}
\psi_{b+1}(y(z); a) \psi_{c+1}(y(P(z));a) \\
&& \cdot (a(a+1))^{n} \prod_{i=2}^n \psi_{b_i+1}(y_i, a) \frac{dx_i}{x_i}  \cdot
\frac{d x(z)}{x(z)} \cdot \frac{dx(P(z))}{x(P(z)} , \\
&& W_{g_1}({a \over a+1} + z, y_A)
= (-1)^{g_1+|A|+1} \sum \langle \tau_b \prod_{i\in A} \tau_{b_i} \cdot T_{g_1}(a)  \rangle_{g_1} \\
&& \qquad \cdot (a(a+1))^{|A|} \psi_{b+1}(y(z); a)   \frac{d x(z)}{x(z)}
 \cdot \prod_{i\in A} \psi_{b_i+1}(y_i; a)  \frac{dx_i}{x_i}, \\
&& W_{g_2}({a \over a+1} + P(z), y_B)
=(-1)^{g_2+|B|+1} \sum \langle \tau_c \prod_{i\in B} \tau_{b_i} \cdot T_{g_2}(a)  \rangle_{g_2} \\
&& \qquad \cdot (a(a+1))^{|B|} \psi_{c+1}(y(P(z)); a) \frac{d x(P(z))}{x(P(z))}
\cdot \prod_{i\in B} \psi_{b_i+1}(y_i; a)  \frac{d x_i}{x_i}.
\een
Because $x(P(z)) = x(z)$,
we actually have:
\be
\frac{dx(P(z))}{x(P(z))} = \frac{d x(z)}{x(z)}.
\ee
For the exceptional terms:
\ben
&& W_0(\frac{a}{a+1}+z, y_i)
=  \frac{d x(z)}{x(z)} \biggl( \frac{x(z) dy_i}{\frac{\pd x(z)}{\pd z}(\frac{a}{a+1}+z-y_i)^2}
- \frac{x(z)dx_i}{(x(z)-x_i)^2} \biggr).
\een
Therefore,
(\ref{eqn:Bou-Mar}) can be rewritten as the following equation:
\begin{multline} \label{eqn:Bou-MarInResidue}
\sum_{b_1 \geq 0} \langle \prod_{i=1}^n \tau_{b_i} \cdot T_g(a)  \rangle_g
\cdot \prod_{i=1}^n \psi_{b_i+1}(y_i; a) \cdot \frac{\frac{\pd x_1}{\pd y_1}}{x_1} \\
= \res_{z=0}  \frac{{1 \over 2}\left( {1 \over z_1 - z} - {1 \over z_1 - P(z)} \right) }
{ \ln y(z) - \ln y(P(z)) } \cdot \frac{dx(z)}{x(z)}  \\
\Big( a(a+1) \sum_{b,c \geq 0} \langle \tau_b \tau_c \prod_{i=2}^n \tau_{b_i} \cdot T_{g-1}(a)  \rangle_{g-1} \\
\cdot \psi_{b+1}(y(z); a) \psi_{c+1}(y(P(z));a) \prod_{i\in [n]_1} \psi_{b_i+1}(y_i; a)  \\
- \sum_{\substack{g_1+g_2=g\\ A\coprod B = [n]_1}}^{stable}
\sum_{b \geq 0} \langle \tau_b \prod_{i\in A} \tau_{b_i} \cdot T_{g_1}(a)  \rangle_{g_1}
\cdot \sum_{c \geq 0} \langle \tau_c \prod_{i\in B} \tau_{b_i} \cdot T_{g_2}(a)  \rangle_{g_2} \\
 \cdot \psi_{b+1}(y(z); a) \psi_{c+1}(y(P(z)); a) \prod_{i \in [n]_1} \psi_{b_i+1}(y_i; a)    \\
- \frac{1}{a(a+1)} \sum_{i=2}^n  \biggl( \frac{x(z) x_i \frac{\pd y_i}{\pd x_i}}{\frac{\pd x(z)}{\pd z}(\frac{a}{a+1}+z-y_i)^2}
- \frac{x(z) x_i}{(x(z)-x_i)^2} \biggr)  \\
 \cdot \sum_{b \geq 0} \langle \tau_b \prod_{j\in [n]_{1i}} \tau_{b_j} \cdot T_{g}(a)  \rangle_{g}
\cdot \psi_{b+1}(y(P(z)); a)
\cdot \prod_{j \in [n]_{1i}} \psi_{b_j+1}(y_j; a)  \\
- \frac{1}{a(a+1)} \sum_{i=2}^n   \biggl( \frac{x(z) x_i \frac{\pd y_i}{\pd x_i}}{\frac{\pd x(z)}{\pd z}(\frac{a}{a+1}+P(z)-y_i)^2}
- \frac{x(z) x_i}{(x(z)-x_i)^2} \biggr)  \\
 \cdot \sum_{b \geq 0} \langle \tau_b \prod_{j\in [n]_{1i}} \tau_{b_j} \cdot T_{g}(a)  \rangle_{g}
\cdot \psi_{b+1}(y(z); a) \cdot \prod_{j \in [n]_{1i}} \psi_{b_j+1}(y_j; a)   \Big).
\end{multline}

\subsection{Change to $t$-coordinates}
Note
\ben
&& {1 \over 2}\left( {1 \over z_1 - z} - {1 \over z_1 - P(z)} \right)
= \sum_{k=1}^{\infty} (\frac{z^k}{z_1^{k+1}} - \frac{P(z)^k}{z_1^{k+1}})
= \sum_{m \geq 1} a_m(z_1) z^m,
\een
where $a_m(z_1) \in z_1^{-2} \bC[z_1^{-1}]$.
Therefore,
each residue term on the right-hand side of (\ref{eqn:Bou-MarInResidue}) lies in $z_1^{-2}\bC[z_1^{-1}]$.
This suggests to define
\be \label{eqn:TinZ}
t = \frac{1}{(a+1)z}.
\ee
It is not hard to check that
\begin{equation}
t(x; a)
= 1 + (a+1) \sum_{n=1}^\infty \frac{\prod_{a=1}^{n-1} (na +a)}{(n-1)!} x^n
= 1+ (a+1)\phi_0(x;a),
\end{equation}
Now we have:
\be \label{eqn:XinT}
x = \frac{a^a}{(a+1)^{a+1}}
(1- \frac{1}{t}) (1 + \frac{1}{a t})^{a}.
\ee
For positive integer $a$,
this defines $x$ as a meromorphic function in $t$ with a pole at $t = 0$.

It is clear that for $b \geq 0$,
\be
\phi_b(x; a) = (x\frac{\pd}{\pd x})^{b} \frac{t(a;x) - 1}{a+1}.
\ee
By easy calculations,
\be \label{eqn:Operator}
x \frac{\pd}{\pd x} = \frac{u(1-u)}{1 - (a+1) u} \frac{\pd}{\pd u}
= \frac{t}{a+1} (t-1)(a t+1) \frac{\pd}{\pd t}.
\ee
Hence under the change of variables $x \mapsto t$,
$\phi_b(x; a)$ becomes
\be
\hat{\xi}_{b}(t; a) = D_t^b \frac{t-1}{a+1}
\ee
for $b \geq 0$,
where
\be
D_t = \frac{1}{a+1} t(t-1)(a t+1) \frac{\pd}{\pd t}.
\ee
It follows that $\hat{\xi}_b(t;a)$ is a polynomial of degree $2b+1$ in $t$,
for example,
\bea
&& \hat{\xi}_0(t; a) = \frac{t-1}{a+1}, \\
&& \hat{\xi}_1(t; a) = \frac{1}{a+1} t(t-1)(a t+1).
\eea
Actually for $b \geq 0$, $\hat{\xi}_{b+1}(t;a) \in t\bC[t]$,
therefore,
$\psi_{b+1}(y(z);a)$ lies in $z^{-1}\bZ[z^{-1}]$.
By (\ref{eqn:Operator}),
we also have
\ben
\psi_{b+1}(y;a) \frac{dx}{x}
= D_t^{b+1}t \cdot \frac{(a+1) dt}{t(t-1)(at+1)}
= - \frac{1}{a+1} \frac{\pd}{\pd t} D_t^bt \cdot \frac{dz}{z^2} \in \frac{dz}{z^2} \bC[z^{-1}].
\een
This means the left-hand side of (\ref{eqn:Bou-MarInResidue}) also lies in $z_1^{-2}\bC[z_1^{-1}]$.

One can also find explicit expressions for $\hat{\xi}_{-1}(t; a)$ and $\hat{\xi}_{-2}(t; a)$.
Note
\be
x\frac{\pd}{\pd x} \phi_{-1}(x; a) = \phi_0(x; a) = \frac{t - 1}{a+1},
\ee
therefore,
\be
 \frac{1}{a+1}  t(t-1)(a t+1) \frac{\pd}{\pd t} \hat{\xi}_{-1}(t; a)
 = \frac{t - 1}{a+1}.
\ee
Integrating once, one gets:
\be \label{eqn:Xi-1inT}
\hat{\xi}_{-1}(t; a) = \ln t - \ln (t+ \frac{1}{a}) - \ln \frac{a}{a+1},
\ee
where we have use
\begin{align}
t(x; a)|_{x=0} & = 1, &
\phi_{-1}(x; a)|_{x=0} & = 0.
\end{align}
One can rewrite (\ref{eqn:Xi-1inT}) as
\be
\hat{\xi}_{-1}(t; a) = - \ln (1-u) = - \ln y.
\ee
Similarly, one can integrate
\be
\frac{1}{a+1}t(t-1)(a t+1) \frac{\pd}{\pd t} \hat{\xi}_{-2}(t;a) = \hat{\xi}_{-1}(t;a)
\ee
to get an expression of $\hat{\xi}_{-2}(t;a)$ in $t$.

\subsection{Residue calculations}

Now we have known that every term in (\ref{eqn:Bou-MarInResidue})
is an element in $z_1^{-2}\bC[z_1^{-1}]$.

\begin{proposition} \label{prop:Residue1}
The residue
\ben
\res_{z=0} \frac{\frac{1}{z_1 - z} - \frac{1}{z_1 - P(z)} }
{\ln y(z) - \ln y(P(z)) } \cdot
\psi_{b+1}(y(z); a) \psi_{c+1}(y(P(z));a)     \frac{d x(z)}{x(z)}
\een
is the principal part in $z_1$ of
\be
\frac{\phi_{b+1}(z_1; a) \psi_{c+1}(P(z_1);a) + \psi_{b+1}(P(z_1); a) \psi_{c+1}(z_1;a)}
{\ln y(z_1) - \ln y(P(z_1)) } \cdot \frac{\frac{\pd x_1}{\pd z_1}}{x_1}.
\ee
\end{proposition}

\begin{proof}
Because $\phi_{b+1}(y(z);a)$ is a polynomial
in $t = \frac{1}{(a+1)z}$,
$$f(z)dz:= \frac{\frac{1}{z_1 - z} - \frac{1}{z_1 - P(z)} }
{\ln y(z) - \ln y(P(z)) } \cdot
\psi_{b+1}(y(z); a) \psi_{c+1}(y(P(z));a)  \frac{d x(z)}{x(z)}$$
is a meromorphic form with a pole at $z=0$,
a simple pole at $z = z_1$ and a simple pole at $z= P(z_1)$.
Here we have assumed $|z_1|$ very small.
By Cauchy's residue theorem,
for sufficiently small $\epsilon > 0$,
\ben
\frac{1}{2\pi i} \int_{|z|=\epsilon} f(z) dz
= \res_{z=0} f(z) dz
+  \res_{z=z_1} f(z) dz  +  \res_{z=P(z_1)} f(z) dz.
\een
The LHS is a   function $r_1(z_1)$  analytic in $z_1$.
At the simple poles we have:
\ben
\res_{z=z_1} f(z) dz
& = & - \frac{\psi_{b+1}(y(z_1); a) \psi_{c+1}(y(P(z_1));a) } {\ln y(z_1) - \ln y(P(z_1)) }
\cdot \frac{\frac{\pd x_1}{\pd z_1}}{x_1},
\een
and
\ben
\res_{z=P(z_1)} f(z) dz
& = & \frac{ \phi_{b+1}(y(P(z_1)); a) \phi_{c+1}(y(z_1);a) }{\ln y(P(z_1)) - \ln y(z_1) }
\cdot \frac{\frac{\pd x}{\pd z}(P(z_1))}{x_1 \cdot P'(P(z_1))}.
\een
Using $P(P(z)) = z$ we get $P'(P(z)) \cdot P'(z) = 1$ and
using $x(P(z)) = x(z)$ we get
$\frac{\pd}{\pd z}x(P(z)) = \frac{\pd x}{\pd z} (P(z)) \cdot P'(z) = x'(z)$.
Combining these two identities,
one gets:
\be
\frac{ \frac{\pd x}{\pd z} (P(z))}{P'(P(z))} = x'(z).
\ee
\end{proof}

In the same fashion one can prove the following:

\begin{proposition}
The residue
\ben
&& \res_{z=0}  \frac{\left( {1 \over z_1 - z} - {1 \over z_1 - P(z)} \right) }
{ \ln y(z) - \ln y(P(z)) } \cdot \frac{dx(z)}{x(z)} \\
&&\biggl( \frac{x(z) x_i \frac{\pd y_i}{\pd x_i}}{\frac{\pd x(z)}{\pd z}(\frac{a}{a+1}+z-y_i)^2}
- \frac{x(z) x_i}{(x(z)-x_i)^2} \biggr) \cdot \phi_{b+1}(y(P(z)); a)
\een
is the principal part in $z_1$ of
\begin{multline}
\frac{1}{ \ln y(z_1) - \ln y(P(z_1)) }  \cdot \frac{\frac{\pd x_1}{\pd z_1}}{x_1} \\
\cdot \biggl[
\biggl( \frac{x_1 \frac{\pd y_1}{\pd x_1} x_i \frac{\pd y_i}{\pd x_i}}{(y_1-y_i)^2}
- \frac{x_1 x_i}{(x_1-x_i)^2} \biggr) \phi_{b+1}(y(P(z_1)); a) \\
+ \biggl( \frac{\frac{\pd y}{\pd x}(P(z_1);a) x_i \frac{\pd y_i}{\pd x_i}}{(y(P(z_1);a)-y_i)^2}
- \frac{x_1x_i}{(x_1-x_i)^2} \biggr) \phi_{b+1}(y(z_1);a) \biggr],
\end{multline}
and the residue
\ben
&& \res_{z=0}  \frac{\left( {1 \over z_1 - z} - {1 \over z_1 - P(z)} \right) }
{ \ln y(z) - \ln y(P(z)) }  \cdot \frac{dx(z)}{x(z)}\\
&& \biggl( \frac{x(z) dy_i}{\frac{\pd x(z)}{\pd z}(\frac{a}{a+1}+P(z)-y_i)^2}
- \frac{x(z)dx_i}{(x(z)-x_i)^2} \biggr) \cdot \phi_{b+1}(y(z); a)
\een
is the principal part in $z_1$ of
\begin{multline}
 \frac{1}{ \ln y(z_1) - \ln y(P(z_1)) }  \cdot \frac{\frac{\pd x_1}{\pd z_1}}{x_1}  \\
 \cdot \biggl[
\biggl( \frac{x_1\frac{\pd y_1}{\pd x_1} x_i\frac{\pd y_i}{\pd x_i}}{(y_1-y_i)^2} - \frac{x_1x_i}{(x_1-x_i)^2} \biggr) \phi_{b+1}(y(P(z_1)); a) \\
+ \biggl( \frac{x_i \frac{\pd y}{\pd x}(P(z_1);a) dy_i}{(y(P(z_1);a)-y_i)^2}
- \frac{x_1x_i}{(x_1-x_i)^2} \biggr) \phi_{b+1}(y(z_1);a) \biggr].
\end{multline}
\end{proposition}

By these results, we see that the Bouchard-Mari\~no recursion is equivalent to:
\begin{multline} \label{eqn:Bou-MarInZ}
\sum_{b_1, \dots b_n \geq 0} \langle \prod_{i=1}^n \tau_{b_i} \cdot T_g(a) \rangle_g \cdot
\prod_{i=1}^n\phi_{b_i+1}(y_i, a)   \cdot \frac{\frac{\pd x_1}{\pd y_1}}{x_1} \\
=   {1 \over 2} \frac{1}{ \ln y(z_1) - \ln y(P(z_1)) }   \cdot \frac{\frac{\pd x_1}{\pd y_1}}{x_1}  \\
 \sum_{b_1, \dots b_n \geq 0} \Big( a(a+1) \sum_{b,c \geq 0} \langle \tau_b \tau_c \prod_{i=2}^n \tau_{b_i} \cdot T_{g-1}(a)  \rangle_{g-1}
\cdot \prod_{i=2}^n \phi_{b_i+1}(y_i; a)  \\
\cdot [\phi_{b+1}(y(z_1); a) \phi_{c+1}(y(P(z_1));a)
+ \phi_{b+1}(y(P(z_1)); a) \phi_{c+1}(y(z_1);a) ]\\
- \sum_{\substack{g_1+g_2=g\\ A\coprod B = [n]_1}}^{stable}
\sum \langle \tau_b \prod_{i\in A} \tau_{b_i} \cdot T_{g_1}(a)  \rangle_{g_1}
\cdot \sum_{b,c \geq 0} \langle \tau_c \prod_{i\in B} \tau_{b_i} \cdot T_{g_2}(a)  \rangle_{g_2}
\cdot \prod_{i=2}^n \phi_{b_i+1}(y_i; a)  \\
 \cdot [\phi_{b+1}(y(z); a) \phi_{c+1}(y(P(z));a)
+ \phi_{b+1}(y(P(z)); a) \phi_{c+1}(y(z);a) ] \\
-2  \sum_{i=2}^n  \sum_{c \geq 0} \langle \tau_c \prod_{j\in [n]_{1i}} \tau_{b_j} \cdot T_{g}(a)  \rangle_{g}
\cdot \prod_{j \in [n]_{1i}} \phi_{b_j+1}(y_j; a)   \\
\frac{1}{a(a+1)} \biggl[
\biggl( \frac{x_1\frac{\pd y_1}{\pd x_1} x_i \frac{\pd y_i}{\pd x_i}}{(y_1-y_i)^2}
- \frac{x_1x_i}{(x_1-x_i)^2} \biggr) \phi_{b+1}(y(P(z_1)); a) \\
+ \biggl( \frac{x_1\frac{\pd y}{\pd x}(P(z_1);a) x_i \frac{\pd y_i}{x_i}}{(y(P(z_1);a)-y_i)^2}
- \frac{x_1x_i}{(x_1-x_i)^2} \biggr) \phi_{b+1}(y(z_1);a) \biggr]  \Big),
\end{multline}
modulo a term analytic in $z_1$.

\subsection{Local Airy curve coordinate}

Following \cite{EMS},
introduce a new coordinate $w$ by
\be \label{eqn:XinW}
x = \frac{a^a}{(a+1)^{a+1}} e^{-w},
\ee
so that the ramification point has $w = 0$.
By (\ref{eqn:XinT}),
\be
e^{-w} = (1-\frac{1}{t}) ( 1 + \frac{1}{a t})^a.
\ee
Therefore,
\be \label{eqn:WinT}
w = - \ln (1-\frac{1}{t}) - a \ln ( 1 + \frac{1}{a t})
= \sum_{n=2}^\infty \frac{1}{n} \biggl( \frac{1}{t^n}
+ \frac{(-1)^n}{a^{n-1}t^n} \biggr)
\ee
for $|t| > 1$.
The leading term is $\frac{a+1}{2a t^2}$.
Let
\be
\begin{split}
v & = \sqrt{ \frac{a+1}{a}}  \frac{1}{t} \biggl(1 + \frac{1}{3}(1-\frac{1}{a}) \frac{1}{t} +
\frac{1}{36} (7 - \frac{5}{a} + \frac{7}{a^2}) \frac{1}{t^2} \\
& + \frac{1}{540} (73 - \frac{48}{a} +\frac{48}{a^2} - \frac{73}{a^3})\frac{1}{t^3} \\
& + \frac{1}{12960}(1331-\frac{842}{a} + \frac{1036}{a^2} - \frac{842}{a^3}
+  \frac{1331}{a^4}) \frac{1}{t^4} + \cdots \biggr)
\end{split}
\ee
that solves $w = \half v^2$.
This expresses $v$ as analytic function in $\frac{1}{t}$ for $\frac{1}{t}$
and $v$ near $0$,
therefore,
by taking the inverse function,
$\frac{1}{t}$ is analytic in $v$:
\be \label{eqn:1/TinV}
\begin{split}
\frac{1}{t} & = \hat{v} (1 - \frac{1}{3}(1-\frac{1}{a})  \hat{v} +
\frac{1}{36} (1 -\frac{11}{a} + \frac{1}{a^2}) \hat{v}^2 \\
& + \frac{1}{270} (1 + \frac{24}{a}  -\frac{24}{a^2} -\frac{1}{a^3}) \hat{v}^3 \\
& + \frac{1}{4320} (1 - \frac{22}{a} +\frac{267}{a^2} - \frac{22}{a^3} + \frac{1}{a^4}) \hat{v}^4
+  \cdots),
\end{split}
\ee
where $\hat{v} = \sqrt{\frac{a}{a+1}} v$, hence
\be \label{eqn:TinVhat}
\begin{split}
t & = \frac{1}{\hat{v}} \biggl( 1 + \frac{1}{3} (1- \frac{1}{a}) \hat{v}
+ \frac{1}{12}(1+\frac{1}{a} + \frac{1}{a^2})\hat{v}^2 \\
& + \frac{1}{135} (2 + \frac{3}{a} - \frac{3}{a^2} - \frac{2}{a^3}) \hat{v}^3 \\
& + \frac{1}{864} (1 + \frac{2}{a} + \frac{3}{a^2} + \frac{2}{a^3} + \frac{1}{a^4}) \hat{v}^4
+ \cdots \biggr)
\end{split}
\ee
is meromorphic in $v$ with a simple pole at $v = 0$.

The involution $P$ in the $z$-coordinate (and the involution $p$ in the $t$-coordinate)
becomes simply
\be
v(p(t)) = - v(t).
\ee
In other words,
\be
p(t) = t(-v; a).
\ee

Because of (\ref{eqn:TinZ}) and (\ref{eqn:1/TinV}),
\be
\begin{split}
z &= \frac{1}{a+1}\frac{1}{t}
= \frac{1}{a+1}  \hat{v} (1 - \frac{1}{3}(1-\frac{1}{a})  \hat{v} +
\frac{1}{36} (1 -\frac{11}{a} + \frac{1}{a^2}) \hat{v}^2 \\
& + \frac{1}{270} (1 + \frac{24}{a}  -\frac{24}{a^2} -\frac{1}{a^3}) \hat{v}^3 \\
& + \frac{1}{4320} (1 - \frac{22}{a} +\frac{267}{a^2} - \frac{22}{a^3} + \frac{1}{a^4}) \hat{v}^4
+  \cdots)
\end{split}
\ee
is analytic in $v$ near $v=0$,
and so a function analytic in $z$ is also analytic in $v$.
Furthermore,
because
\be
\frac{d z}{d v} = \sqrt{\frac{a}{(a+1)^3}} \neq 0,
\ee
we have
\be
\res_{z=0} f(z) = \res_{v=0} f(z(v)).
\ee

\subsection{The odd part of $\xi_b(v;a)$}

For $b \geq 0$,
\ben
\xi_b(v;a) = (-\frac{1}{v} \frac{d}{dv})^b \frac{t-1}{a+1}
\een
is a memermorphic function in $v$ with a pole of order $2b+1$ at $v= 0$.
Furthermore,
if we write
\be
t = t_o + t_e,
\ee
where $t_o$ and $t_e$ are odd and even functions in $v$ respectively:
\bea
t_o & = & \frac{1}{\hat{v}}
+ \frac{1}{12}(1+\frac{1}{a} + \frac{1}{a^2})\hat{v}
+ \frac{1}{864} (1 + \frac{1}{a} + \frac{1}{a^2} )^2 \hat{v}^3
+ \cdots, \label{eqn:T-odd}  \\
t_e & = & \frac{1}{3} (1- \frac{1}{a})
+ \frac{1}{135} (2 + \frac{3}{a} - \frac{3}{a^2} - \frac{2}{a^3}) \hat{v}^2
+ \cdots,
\eea
then under the repeated action of $\frac{1}{v}\frac{d}{d v}$,
$t_e$ will only contribute nonnegative powers of $v$.
It follows that the odd part of $\hat{\xi}_b$ is completely determined by $t_o$.
In particular for $b > 0$ the principal part of $\hat{\xi}_b$  is completely determined by $t_o$
and $\hat{\xi}_{-b}$ is analytic in $v$.

\subsection{Bouchard-Mari\~no Conjecture in $v$-coordinates}
Note
\be \label{eqn:DerXZ}
\frac{\frac{\pd x_1}{\pd y_1}}{x_1}
= \frac{\frac{\pd x_1}{\pd v_1}}{x_1} \cdot \frac{1}{\frac{\pd z_1}{\pd v_1}}
= -v_1 \frac{1}{\frac{\pd z_1}{\pd v_1}}.
\ee
Now by (\ref{eqn:1/TinV}) one can see that (\ref{eqn:Bou-MarInZ}) is equivalent to
\begin{multline} \label{eqn:Bou-MarInV}
\sum_{b_1, \dots b_n \geq 0} \langle \prod_{i=1}^n \tau_{b_i} \cdot T_g(a) \rangle_g \cdot
\prod_{i=1}^n\xi_{b_i+1}(v_i, a)   \\
=   {1 \over 2} \frac{1}{- \xi_{-1}(v_1;a) + \xi_{-1}(-v_1;a) }
 \Big( a(a+1) \sum_{b,c \geq 0} \langle \tau_b \tau_c \prod_{i=2}^n \tau_{b_i} \cdot T_{g-1}(a)  \rangle_{g-1}  \\
\cdot [\xi_{b+1}(v_1; a) \xi_{c+1}(-v_1);a) + \xi_{b+1}(-v_1); a) \xi_{c+1}(v_1;a) ]
\cdot \prod_{i=2}^n \xi_{b_i+1}(v_i, a)  \\
- \sum_{\substack{g_1+g_2=g\\ A\coprod B = [n]_1}}^{stable}
\sum \langle \tau_b \prod_{i\in A} \tau_{b_i} \cdot T_{g_1}(a)  \rangle_{g_1}
\cdot \sum_{b,c \geq 0} \langle \tau_c \prod_{i\in B} \tau_{b_i} \cdot T_{g_2}(a)  \rangle_{g_2} \\
\cdot [\xi_{b+1}(v_1; a) \xi_{c+1}(-v_1);a) + \xi_{b+1}(-v_1); a) \xi_{c+1}(v_1;a) ]
 \cdot \prod_{i=2}^n \xi_{b_i+1}(v_i, a)  \\
-2  \sum_{i=2}^n  \sum_{c \geq 0} \langle \tau_c \prod_{j\in [n]_{1i}} \tau_{b_j} \cdot T_{g}(a)  \rangle_{g}
\cdot \prod_{j \in [n]_{1i}} \xi_{b_j+1}(v_j; a)   \\
\biggl[
\biggl( \frac{x_1x_i\frac{\pd y}{\pd x}(z_1;a) \frac{\pd y_i}{\pd x_i}}{(y_1-y_i)^2} - \frac{x_1x_i}{(x_1-x_i)^2} \biggr)
\xi_{b+1}(-v_1; a) \\
+ \biggl( \frac{x_1x_i \frac{\pd y}{\pd x}(P(z_1);a) \frac{\pd y_i}{\pd x_i}}{(y(P(z_1);a)-y_i)^2}
- \frac{x_1x_i}{(x_1-x_i)^2} \biggr) \xi_{b+1}(v_1;a) \biggr]   \Big),
\end{multline}
modulo a term with at most a simple pole at $0$ in $v_1$.
This is because in determining the LHS from the RHS,
only terms with degrees $-3$ and lower in $v_1$ are involved.
For  $f(v_1,\dots, v_n)$ write
$$f(v_1, \dots,v_n) ^o = \half (f(v_1, \dots, v_n) -   f(-v_1, \dots, v_n) ).$$
Noticing
$$\biggl( \frac{x_1x_i}{(x_1-x_i)^2} \biggr)^o =0,
$$
we get from (\ref{eqn:Bou-MarInV}):
\begin{multline} \label{eqn:Bou-MarInVOdd}
\sum_{b_1, \dots b_n \geq 0} \langle \prod_{i=1}^n \tau_{b_i} \cdot T_g(a) \rangle_g \cdot
 \xi_{b_i+1}^o(v_i, a) \prod_{i=2}^n \xi_{b_i+1}(v_i, a)  \\
=   {1 \over 2} \frac{1}{\xi_{-1}^o(v_1;a) }
 \Big( a(a+1) \sum_{b,c \geq 0} \langle \tau_b \tau_c \prod_{i=2}^n \tau_{b_i} \cdot T_{g-1}(a)  \rangle_{g-1}  \\
\cdot \xi^o_{b+1}(v_1; a) \xi^o_{c+1}(v_1;a)
\cdot  \prod_{i=2}^n \xi_{b_i+1}(v_i, a)  \\
- \sum_{\substack{g_1+g_2=g\\ A\coprod B = [n]_1}}^{stable}
\sum \langle \tau_b \prod_{i\in A} \tau_{b_i} \cdot T_{g_1}(a)  \rangle_{g_1}
\cdot \sum_{b,c \geq 0} \langle \tau_c \prod_{i\in B} \tau_{b_i} \cdot T_{g_2}(a)  \rangle_{g_2} \\
\cdot \xi_{b+1}^o(v_1; a) \xi_{c+1}^o(v_1;a)
 \cdot \prod_{i=2}^n \xi_{b_i+1}(v_i, a) \\
+ \frac{2}{a(a+1)}  \sum_{i=2}^n  \sum_{c \geq 0} \langle \tau_c \prod_{j\in [n]_{1i}} \tau_{b_j} \cdot T_{g}(a)  \rangle_{g}
\cdot \prod_{j \in [n]_{1i}} \xi_{b_j+1}(v_j; a)   \\
\cdot
\biggl( \frac{x_1x_i\frac{\pd y}{\pd x}(z_1;a) \frac{\pd y_i}{\pd x_i}}{(y_1-y_i)^2} \biggr)^o \xi^o_{b+1}(v; a)    \Big),
\end{multline}
modulo a term with at most a simple pole at $0$ in $v_1$.
This is an equivalent reformulation of the Bouchard-Mari\~no recursion that we will establish in the next section.

\section{Proof of Bouchard-Mari\~no Conjecture}

In this section we derive (\ref{eqn:Bou-MarInVOdd}) from the cut-and-join equation.

\subsection{The cut-and-join equation for triple Hodge integrals of KLMV type}

For a partition $\mu = (\mu_1, \dots, \mu_{l(\mu)})$ of $d>0$,
consider the triple Hodge integral:
\begin{eqnarray*}
\cC_{g, \mu}(a)& = & - \frac{\sqrt{-1}^{|\mu|+l(\mu)}}{|\Aut(\mu)|}
 (a(a+1))^{l(\mu)-1}
\prod_{i=1}^{l(\mu)}\frac{ \prod_{a=1}^{\mu_i-1} (\mu_i a+a)}{(\mu_i-1)!} \\
&& \cdot \int_{\Mbar_{g, l(\mu)}}
\frac{\Lambda^{\vee}_g(1)\Lambda^{\vee}_g(-a-1)\Lambda_g^{\vee}(a)}
{\prod_{i=1}^{l(\mu)}(1- \mu_i \psi_i)},
\een
and their generating series
\ben
\cC_{\mu}(\lambda; a) & = & \sum_{g \geq 0} \lambda^{2g-2+l(\mu)}\cC_{g, \mu}(a)
\end{eqnarray*}
and\begin{eqnarray*}
\cC(\lambda; a; p) & = & \sum_{|\mu| \geq 1} \cC_{\mu}(\lambda;a)p_{\mu},
\end{eqnarray*}
where $p_\mu = \prod_{i=1}^{l(\mu)} p_{\mu_i}$.
They satisfy the following cut-and-join equation \cite{Zho1, LLZ}:
\begin{eqnarray*}
 \frac{\partial \cC}{\partial a}
= \frac{\sqrt{-1}\lambda}{2} \sum_{i, j\geq 1} \left(ijp_{i+j}
\frac{\partial^2\cC}{\partial p_i\partial p_j}
+ ijp_{i+j}\frac{\partial \cC}{\partial p_i}\frac{\partial\cC}{\partial p_j}
+ (i+j)p_ip_j\frac{\partial \cC}{\partial p_{i+j}}\right). \label{eqn:CutJoin}
\end{eqnarray*}

\subsection{Symmetrization}
One can also define
\ben
&& \cC_g(p;a)
= \sum_{\mu} \cC_{g, \mu}(a) p_{\mu}.
\een
Because $\cC_g(a;p)$ is a formal power series in $p_1, p_2, \dots, p_n, \dots$,
for each $n$,
one can obtain from it a formal power series $\Phi_{g,n}(x_1, \dots, x_n; a)$
by applying the following linear symmetrization operator \cite{Gou-Jac-Vai, Che}:
$$p_{\mu} \mapsto (\sqrt{-1})^{-(n+|\mu|)}
\delta_{l(\mu), n}\sum_{\sigma \in S_n}
x_{\sigma(1)}^{\mu_1} \cdots x_{\sigma(n)}^{\mu_n}.$$
From the definition,
we have for $2g-2+n > 0$,
\ben
&& \Phi_{g,n}(a; x_1, \dots, x_n) \\
& = &  -(a(a+1))^{n-1} \sum_{b_1, \dots, b_n \geq 0}
\cor{\tau_{b_1} \cdots \tau_{b_n} T_g(a)}_g
\prod_{i=1}^n \phi_{b_i}(x_i; a),
\een
where
\ben
&& \cor{\tau_{b_1} \cdots \tau_{b_n} T_g(a)}_g
= \int_{\Mbar_{g,n}} \prod_{i=1}^n \psi_i^{b_i} \cdot
\Lambda_g^{\vee}(1)\Lambda_g^{\vee}(a)\Lambda_g^{\vee}(-1-a), \\
&& \phi_b(x;a) = \frac{1} {a}\sum_{m\geq 1}\frac{\prod_{j=0}^{m-1}(ma+j)}{m!} m^b x^m.
\een
We have two exceptional cases.
For $(g,n) = (0,1)$,
\be \label{eqn:Phi01}
\Phi_{0,1}(x_1;a) = - \sum_{m=1}^{\infty} \frac{\prod_{j=1}^{m-1} (m a +j)}{(m-1)!} m^{-2} x_1^m
= - \phi_{-2}(x_1;a);
\ee
for $(g, n) = (0, 2)$,
\be \label{eqn:Phi02}
\Phi_{0,2}(x_1, x_2; a)
= -a(a+1) \sum_{m_1, m_2 \geq 1} \prod_{i=1}^2 \frac{\prod_{j=1}^{m_i-1} (m_ia+j)}{(m_i-1)!}
\cdot \frac{x_1^{m_1}x_2^{m_2}}{m_1+m_2}.
\ee

By the analysis in \cite{Gou-Jac-Vai, Che},
the symmetrized cut-and-join equation of triple Hodge integrals of KLMV type is
\be \label{eqn:CJinX}
\begin{split}
&  \frac{\pd}{\pd a} \Phi_{g,n}(x_{[n]}; a) \\
=&-\frac{1}{2}\sum_{i=1}^n z_1 \frac{\pd}{\pd z_1}
z_2 \frac{\pd}{\pd z_2}\Phi_{g-1,n+1}(z_1, z_2,x_{[n]_i}; a)|_{z_1, z_2=x_i}\\
-& \frac{1}{2} \sum_{i=1}^n \sum_{\substack{g_1+g_2 = g \\ A \coprod B = [n]_i} }^{stable}
x_i\frac{\pd}{\pd x_i} \Phi_{g_1, |A|+1}(x_i, x_A; a) \\
& \cdot
x_i \frac{\pd}{\pd x_i} \Phi_{g_2, |B|+1}(x_i, x_B; a) \\
-& \sum_{i=1}^n
x_i\frac{\pd}{\pd x_i} \Phi_{0, 1}(x_i; a) \cdot
x_i \frac{\pd}{\pd x_i} \Phi_{g, n}(x_{[n]}; a) \\
-& \sum_{i=1}^n \sum_{j \in [n]_i}
(x_i\frac{\pd}{\pd x_i} \Phi_{0, 2}(x_i, x_j; a) - \frac{x_j}{x_i-x_j}) \cdot
x_i \frac{\pd}{\pd x_i} \Phi_{g, n-1}(x_{[n]_j}; a).
\end{split}
\ee
It is clear that
\be \label{eqn:W=dPhi}
W_g(y_1, \dots, y_n;a)=(-1)^{g+n-1}  \prod_{i=1}^n \frac{\pd}{\pd x_i} \Phi_{g,n}(x, x_{[n]}; a)
\cdot \prod_{i=1}^n dx_i.
\ee
So it is natural to derive Bouchard-Mari\~no recursion from the symmetrized cut-and-join equation.

\subsection{The initial values}
By (\ref{eqn:Phi01}),
we have
\be
x \frac{\pd}{\pd x} \Phi_{0,1}(x;a)
= - \phi_{-1}(x;a) = - \xi_{-1}(v;a) = - \ln y.
\ee
From this one then verifies (\ref{eqn:W0y}).
We also need to find similar expressions for $\Phi_{0,2}(x_1, x_2;a)$. By (\ref{eqn:Phi02}),
\begin{multline} \label{eqn:EqnPhi(02)}
\big(x_1\frac{\pd}{\pd x_1} + x_2\frac{\pd}{\pd x_2}\big) \Phi_{0,2}(x_1, x_2;a)
= -a (a+1) \frac{t_1-1}{a+1} \cdot \frac{t_2-1}{a+1} \\
= - a(a+1) \frac{y_1-1}{(a+1) y_1 - a} \cdot \frac{y_2-1}{(a+1) y_2 - a}.
\end{multline}
Note if $A(x_1, x_2) = \sum_{m_1, m_2 \geq 1} A_{m_1,m_2} x_1^{m_1}x_2^{m_2}$
satisfies
\be
(x_1\frac{\pd}{\pd x_1} + x_2\frac{\pd}{\pd x_2}) A(x_1,x_2)
=  \sum_{m_1, m_2 \geq 1} B_{m_1,m_2} x_1^{m_1}x_2^{m_2},
\ee
then $A(x_1, x_2)$ is uniquely determined by
\be
A_{m_1,m_2} = \frac{1}{m_1+m_2} B_{m_1, m_2}.
\ee
Therefore,
one can verify that:
\be
\Phi_{0,2}(x_1, x_2; a)=-\ln(\frac{y_2-y_1}{x_1-x_2})
+ \ln \frac{1-y_1}{x_1}+\ln \frac{1-y_2}{x_2}.
\ee
Indeed,
differentiating this equation one gets:
\be \label{eqn:x1DPhi(02)}
x_1\frac{\pd}{\pd x_1} \Phi_{0,2}(x_1,x_2; a)
=-\frac{1}{y_1-y_2} x_1 \frac{\pd y_1}{\pd x_1} + \frac{x_2}{x_1-x_2}
+ \frac{1}{y_1-1} x_1\frac{\pd y_1}{\pd x_1}.
\ee
One can use (\ref{eqn:Operator}) to get
$$x_1 \frac{\pd y_1}{\pd x_1} = \frac{y_1^2-y_1}{(a+1)y_1-a},$$
and so
\ben
x_1\frac{\pd}{\pd x_1} \Phi_{0,2}(x_1,x_2;a)
= -(\frac{1}{y_1-y_2} - \frac{1}{y_1-1} )\frac{y_1^2-y_1}{(a+1)y_1-a}+\frac{x_2}{x_1-x_2}.
\een
One gets $x_2\frac{\pd}{\pd x_2} \Phi_{0,2}(x_1,x_2;\tau)$ by switching $1$ and $2$,
then (\ref{eqn:EqnPhi(02)}) can be checked.

One can also take $x_2\frac{\pd}{\pd x_2}$ on both sides of (\ref{eqn:x1DPhi(02)}) to get:
\be
\frac{\pd}{\pd x_1} \frac{\pd}{\pd x_2} \Phi_{0,2}(x_1,x_2;a)
=-\frac{1}{(y_1-y_2)^2} \frac{\pd y_1}{\pd x_1} \frac{\pd y_2}{\pd x_2} + \frac{1}{(x_1-x_2)^2}.
\ee
By this one verifies (\ref{eqn:W0y1y2}).

\subsection{Symmetrized cut-and-join equation in the $v$-coordinates}

Recall for $2g-2+n> 0$,
\be
\Phi_{g,n}(x_{[n]};a) =   -(a(a+1))^{n-1} \sum_{b_i \geq 0}
\cor{\prod_{i=1}^n \tau_{b_i} \cdot T_g(a)}_g
\prod_{i=1}^n \xi_{b_i}(v_i; a).
\ee
It is straightforward to check that
\be \label{eqn:OperatorsInXWV}
x \frac{\pd}{\pd x} = - \frac{\pd}{\pd w} = - \frac{1}{v} \frac{\pd}{\pd v}.
\ee
Therefore,
\be
\begin{split}
& x_j \frac{\pd}{\pd x_j} \Phi_{g,n}(x_{[n]};a)
= - \frac{1}{v_j} \frac{\pd}{\pd v_j}\Phi_{g,n}(x_{[n]};a) \\
= & (a(a+1))^{n-1} \sum_{b_1, \dots, b_n \geq 0}
\cor{\prod_{i=1}^n \tau_{b_i} \cdot T_g(a)}_g
\prod_{i=1}^n \xi_{b_i+\delta_{ij}}(v_i; a).
\end{split}
\ee

By (\ref{eqn:x1DPhi(02)}),
\be
x_1 \frac{\pd}{\pd x_1} \Phi_{0,2}(x_1, x_2;a) - \frac{x_2}{x_1-x_2}
= \big(\frac{1}{y_1-1} -\frac{1}{y_1-y_2} \big) x_1\frac{\pd y_1}{\pd x_1}.
\ee

Now the symmetrized cut-and-join equation (\ref{eqn:CJinX}) can be written as:
\ben
&& -\frac{\pd}{\pd a} \biggl( (a(a+1))^{n-1} \sum_{b_i \geq 0}
\cor{\prod_{i=1}^n \tau_{b_i} \cdot T_g(a)}_g
\prod_{i=1}^n \xi_{b_i}(v_i; a) \biggr) \\
& = & \half \sum_{i=1}^n (a(a+1))^{n} \sum_{b,c, b_i \geq 0}
\cor{\tau_b\tau_c \prod_{j \in [n]_i} \tau_{b_j} \cdot T_{g-1}(a)}_{g-1} \xi_{b+1}(v_i) \xi_{c+1}(v_i)
\prod_{j \in [n]_i} \xi_{b_j}(v_j; a) \\
& - & \frac{1}{2} \sum_{i=1}^n \sum_{\substack{g_1+g_2 = g \\ A \coprod B = [n]_i} }^{stable}
(a(a+1))^{|A|} \cor{\tau_b \prod_{j \in A} \tau_{b_j} \cdot T_{g_1}(a)}_{g_1} \xi_{b+1}(v)
\prod_{j \in A} \xi_{b_j}(v_j; a) \\
&& \cdot
(a(a+1))^{|B|} \cor{\tau_c \prod_{j \in B} \tau_{b_j} \cdot T_{g_2}(a)}_{g_2} \xi_{c+1}(v)
\prod_{j \in B} \xi_{b_j}(v_j; a) \\
& - & \sum_{i=1}^n \xi_{-1}(v_i; a) \cdot  (a(a+1))^{n-1} \sum_{b_j \geq 0}
\cor{\prod_{j=1}^n \tau_{b_j} \cdot T_g(a)}_g
\prod_{j=1}^n \xi_{b_j+\delta_{ij}}(v_j; a) \\
& + & \sum_{i=1}^n \sum_{j \in [n]_i}
\big(\frac{1}{y_i-1} -\frac{1}{y_i-y_j} \big) x_i \frac{\pd y_i}{\pd x_i} \\
&& \cdot (a(a+1))^{n-2}
 \sum_{b_1, \dots, \hat{b}_j, \dots, b_n \geq 0}
\cor{\prod_{k \in [n]_j} \tau_{b_k} \cdot T_g(a)}_g
 \xi_{b_i+1}(v_i; a) \prod_{k \in [n]_{ij}} \xi_{b_k}(v_k; a).
\een
We regard both sides of this equation as meromorphic functions in $v_1$,
take the principal parts and then take only the even powers in $v_1$.
The left-hand side has no contribution,
so we get:
\ben
&&  \xi^o_{-1}(v_1; a) \cdot  \sum_{b_j \geq 0}
\cor{\prod_{j=1}^n \tau_{b_j} \cdot T_g(a)}_g
\xi_{b_1+1}^o(v_1;a) \prod_{j=2}^n \xi_{b_j}(v_j; a) \\
& = & \half a(a+1) \sum_{b,c, b_i \geq 0}
\cor{\tau_b\tau_c \prod_{j \in [n]_1} \tau_{b_j} \cdot T_{g-1}(a)}_{g-1} \xi^o_{b+1}(v_1) \xi^o_{c+1}(v_1)
\prod_{j \in [n]_1} \xi_{b_j}(v_j; a) \\
& - & \frac{1}{2} \sum_{\substack{g_1+g_2 = g \\ A \coprod B = [n]_i} }^{stable}
 \cor{\tau_b \prod_{j \in A} \tau_{b_j} \cdot T_{g_1}(a)}_{g_1}
\cor{\tau_c \prod_{j \in B} \tau_{b_j} \cdot T_{g_2}(a)}_{g_2} \\
&& \cdot \xi_{b+1}^o(v_1) \xi_{c+1}^o(v_1) \prod_{j \in [n]_1} \xi_{b_j}(v_j; a) \\
& + & \frac{1}{a(a+1)} \sum_{j \in [n]_1}
\biggl[ \big(\frac{1}{y_1-1} -\frac{1}{y_1-y_j} \big) x_1 \frac{\pd y_1}{\pd x_1} \biggr]^o \\
&& \cdot
 \sum_{b_1, \dots, \hat{b}_j, \dots, b_n \geq 0}
\cor{\prod_{k \in [n]_j} \tau_{b_k} \cdot T_g(a)}_g
 \xi^o_{b_1+1}(v_1; a) \prod_{k \in [n]_{1j}} \xi_{b_k}(v_k; a)
\een
modulo terms analytic in $v_1$.
One gets (\ref{eqn:Bou-MarInVOdd}) by taking $\prod_{i=2}^n x_i \frac{\pd}{\pd x_i}$
on both sides then dividing both sides by $\xi_{-1}^o(v;a)$
of this equation.
This completes the proof of Theorem 1.

\vspace{.1in}
{\em Acknowledgements}.
The author thanks Professor Kefeng Liu for bringing \cite{Bou-Mar} and \cite{Che} to his attentions.
This research is partially supported by two NSFC grants (10425101 and 10631050)
and a 973 project grant NKBRPC (2006cB805905).


\begin{thebibliography}{999}

\bibitem{AKMV}
M. Aganagic, A. Klemm, M. Mari\~no, C. Vafa,
{\em The topological vertex},
Commun. Math. Phys. {\bf 254} (2005), 425-478, arXiv:hep-th/0305132.

\bibitem{BCOV}
M. Bershadsky, S. Cecotti, H. Ooguri, C. Vafa,
{\em Kodaira-Spencer theory of gravity
and exact results for quantum string amplitudes},
Commun. Math. Phys. {\bf 165} (1994), 311-427, arXiv:hep-th/9309140.

\bibitem{BEMS} G.~Borot, B.~Eynard,
M.~Mulase,B.~Safnuk, {Hurwitz numbers, matrix models and topological recursion},
arXiv:0906.1206.

\bibitem{BKMP}
V. Boucharda, A. Klemmb, Marcos Mari\~no,  S. Pasquetti,
{\em Remodeling the B-model},
arXiv:0709.1453.

\bibitem{Bou-Mar}
V. Bouchard, M. Mari\~no,
{\em Hurwitz numbers, matrix models and enumerative geometry},
arXiv:0709.1458.

\bibitem{Che}
L. Chen,
{\em  Symmetrized cut-join equation of Marino-Vafa formula},
arXiv:0709.1738.

\bibitem{Che2}
L. Chen,
{\em Bouchard-Klemm-Mari\~no-Pasquetti Conjecture for $\bC^3$},
arXiv:0910.3739.

\bibitem{ELSV} T. Ekedahl, S. Lando, M. Shapiro, A. Vainshtein.
{\em Hurwitz numbers and intersections on moduli spaces of curves},
Invent. Math. {\bf 146} (2001), 297-327.

\bibitem{EMS}
B. Eynard, M. Mulase, B. Safnuk
{\em The Laplace transform of the cut-and-join equation and the Bouchard-Marino conjecture on Hurwitz numbers},
arXiv:0907.5224.


\bibitem{EO}
B. Eynard, N. Orantin,
{\em Invariants of algebraic curves and topological expansion},
arXiv:math-ph/0702045.

\bibitem{Gou-Jac-Vai} I.P. Goulden, D. M. Jackson,A. Vainshtein,
{\em The number of ramified coverings of the sphere by the torus and surfaces
of higher genera}, Ann. Combinatorics {\bf 4} (2000), 27-46.

\bibitem{Kat-Liu}
S.~Katz, C.-C.~Liu,
{\em Enumerative geometry of stable maps with Lagrangian boundary condtions and multiple covers
of the disc},
Adv. Theor. Math. Phys. {\bf 5} (2001), 1-49.

\bibitem{Li-Zha-Zhe} A.M. Li, G. Zhao, Q. Zheng,
{\em The number of ramifed coverings of a Riemann surface by Riemann surface},
Comm.Math.Phys. {\bf 213} (2000), 685-696.

\bibitem{LLZ}  C.-C. Liu, K. Liu, J. Zhou,
{\em A proof of a conjecture of Mari\~no-Vafa on Hodge Integrals}, J. Differential Geom. {\bf 65} (2003),
289-340.

\bibitem{LLZ2} C.-C. Liu, K. Liu, J. Zhou, {\em A formula of two-partition Hodge integrals},
J. Amer. Math. Soc. {\bf 20} (2007), no. 1, 149-184.

\bibitem{LLLZ}
J. Li, C.-C. Liu, K. Liu, J. Zhou,
{\em A mathematical theory of the topological vertex},
arXiv:math/0408426.

\bibitem{Mar}
M. Mari\~no,
{\em Open string amplitudes and large order behavior in topological string theory},
arXiv:hep-th/0612127.

\bibitem{Mar-Vaf}
M. Mari\~no, C. Vafa,
{\em Framed knots at large N}, Orbifolds in mathematics and physics (Madison, WI,
2001), 185-204, Contemp. Math., 310, Amer. Math. Soc., Providence, RI, 2002.

\bibitem{Oko-Pan} A. Okounkov, R. Pandharipande,
{\em Hodge integrals and invariants of the unknot},
Geometry $\&$ Topology, Vol. {\bf  8} (2004), Paper no. 17, 675-699.

\bibitem{Zho1} J. Zhou, {\em Hodge integrals, Hurwitz numbers, and symmetric groups}, math.AG/0308024.

\bibitem{Zho2} J. Zhou, {\em   A conjecture on Hodge integrals}, arXiv:math/0310282.

\bibitem{Zho3} J. Zhou, {\em  Local mirror symmetry for the topological vertex},
in preparation.

\end{thebibliography}
\end{document}